\documentclass{article}
\usepackage{amsmath,amssymb,amsthm,amsfonts}
\makeatletter
\renewcommand\section{\@startsection {section}{1}{\z@}%
                                   {-3.5ex \@plus -1ex \@minus -.2ex}%
                                   {2.3ex \@plus.2ex}%
                                   {\normalfont\Large\bfseries
                                    \setcounter{equation}{0}
                                    \setcounter{thm}{0}}}
\makeatother
\newtheorem{thm}{Theorem}
\newtheorem{lem}[thm]{Lemma}

\newtheorem{defn}[thm]{Definition}

\newcommand{\q}{\quad}
\newcommand{\qq}{\qquad}
\newcommand{\h }{\hspace*{18pt}}
\newcommand{\CommaBin}{\mathbin{\raisebox{0.5ex}{,}}}
\title{\textbf{Existence Results By The Method Of Upper And Lower Solutions For Volterra Integral Equation On Time Scales}}

\author{\small \textbf{Alaa E.\ Hamza}$^{1}$, \textbf{and Ahmed G.~Ghallab}$^{2}$  \\
\small$^\textbf{{1}}$ Department of Mathematics, Faculty of Science, Cairo University, Giza, Egypt.\\
\small E-mail: hamzaaeg2003@yahoo.com\\
\small$^\textbf{{2}}$ Department of Mathematics, Faculty of  Science,\\
 \small {Fayoum University, Fayoum, Egypt.}\\
 \small E-mail: agg00@fayoum.edu.eg}
\date{}

\begin{document}
 \thispagestyle{empty} \maketitle
 \begin{abstract}
 In this article, we investigate the method of upper and lower solutions for Volterra integral equation of the first kind on arbitrary time scale $\mathbb{T}$. We establish some existence results in a certain sector. Moreover, monotone iterative technique is used to obtain maximal and minimal solutions of the considered equation.
 \end{abstract}

\section{Introduction}

\h The method of upper and lower solution is a widely used tool in investigating qualitative properties for many classes of dynamical systems, see for instance \cite{Ladde},\cite{Zhang},\cite{Tong}. The main advantage of this method is the ability to obtain a sector where solutions of considered dynamical systems lie inside it. The upper and lower solutions form  as the upper and lower bound for that sector. Moreover, by using iteration scheme, called the monotone iterative technique, we can improve these bounds of the obtained sector to obtain extremal solutions.

 Throughout this article we shall use the method of upper and lower solution coupled with the method of iterative technique to establish existence and uniqueness of solutions to a certain integral equation of Volterra type on arbitrary time scale $\mathbb{T}$.

%The method of upper and lower solutions is a well-known tool that has been used not only to prove results of the existence of solutions for many classes of boundary value problems involving ordinary and partial differential equations, but also to obtain a sector in which those solutions must remain, and hence some useful a priori information about the solutions.

%\h The method of upper and lower solution is a well-known tool that has been used not only to establish existence results for a variety of dynamic equations on time scales, but also to obtain a \emph{sector} in which those solution must remain, and hence useful a priori information about the solutions. %See \cite{Ladde, Zhang, Tong}. %Also, we consider the existence of minimal and maximal solutions by establishing the method of monotone iterative technique. %The upper and lower solutions that generates the sector serve as upper and lower bounds for solutions which can be improved by monotone iterative techniques.
 We shall consider the following integral equation:
 \begin{equation}\label{eq41}
   x(t)=f(t)+\int_a^tk(t,s,x(s))\Delta s,\q  t\in I_\mathbb{T};
 \end{equation}
where $I_\mathbb{T}:= [a,b]\cap\mathbb{T}$ is a time scale interval, $f:I_\mathbb{T}\rightarrow\mathbb{R}$ , $k:I_\mathbb{T}\times I_\mathbb{T}\times \mathbb{R}\rightarrow \mathbb{R}$, and $x$ is the unknown function.

This article is organized as follows. Some basic concepts and notations of calculus on time scales are given in Section \textbf{2}. Section \textbf{3}, we investigate the existence and uniqueness of the solutions of equation \eqref{eq41} within the sector determined by the upper and lower solutions. In Section \textbf{4}, we use monotone iterative technique to establish a result about the extremal solutions of equation \eqref{eq41}.
\section{Preliminaries}
In this section we introduce some definitions, notations, and preliminary results which will be used throughout this article. For more details see \cite{Hilger,Boh1}.

\begin{defn}
A time scale $\mathbb{T}$ is a nonempty closed subset of the real numbers $\mathbb{R}$.
\end{defn}

%The reals $\mathbb{R}$, the integers $\mathbb{Z}$, and the nonnegative integers $\mathbb{N}$ are examples of time scales.

\begin{defn}
The mappings $\sigma,\rho:\mathbb{T}\rightarrow \mathbb{T}$ defined by $\sigma(t)=\inf\{s\in \mathbb{T}: s>t\}$, and $\rho(t)=\sup\{s\in \mathbb{T}: s<t\}$ are called the jump operators.
\end{defn}
If $\mathbb{T}$ has a left scattered maximum $m$, then $\mathbb{T}^\kappa=\mathbb{T}-{m}$.
\begin{defn}
A function $f:\mathbb{T} \rightarrow \mathbb{R}$ is said to be delta differentiable at the point $t\in \mathbb{T}$ if there exist a number $f^\Delta(t)$ with the property that given any $\epsilon >0$ there is a neighborhood $U$ of $t$ with $\|[f(\sigma(t))-f(s)]-f^\Delta(t)[\sigma(t)-s]\|\leq \varepsilon |\sigma(t)-s|$ for all $s\in U$. The function $f^\Delta(t)$ is the delta derivative of $f$ at $t$.
\end{defn}

For $\mathbb{T}=\mathbb{R}$, we have $f^\Delta(t)=f'(t)$, and for $\mathbb{T}=\mathbb{Z}$, we have $f^\Delta(t)=\Delta f(t)= f(t+1)-f(t)$.

\begin{defn}
A function $F:\mathbb{T}\rightarrow\mathbb{R}$ is called an \emph{antiderivative} of $f:\mathbb{T}\rightarrow\mathbb{R}$ provided $F^\Delta(t)=f(t)$ for all $t\in \mathbb{T}^\kappa$. The $\Delta$-integral of $f$ is defined  by
$$
\int_r^s f(t)\Delta t= F(s)-F(r), \q \text{for all } r,s \in \mathbb{T}.
$$
\end{defn}

If $\mathbb{T}=\mathbb{R}$, then $\displaystyle\int_a^b f(s)\Delta s=\displaystyle\int_a^b f(s)ds$, while if $\mathbb{T}=\mathbb{Z}$, then $\displaystyle\int_a^b f(s)\Delta s=\displaystyle\sum_{s=a}^{b-1}f(s)$.

\begin{defn}
A function $f:\mathbb{T}\rightarrow\mathbb{R}$ is called  right-dense continuous (rd-continuous)  if $f$ is continuous at every right-dense point $t\in \mathbb{T}$ and the left-sided limits exist (i.e finite) at every left-dense point $t\in \mathbb{T}$. The family of all rd-continuous functions from $ \mathbb{T}$ to $\mathbb{R}$ is denoted by $ C_{rd}(\mathbb{T};\mathbb{R})$.
\end{defn}

The family  of all regressive functions is denoted by
$$
\mathcal{R}:= \Big\{p\in C_{rd}(\mathbb{T};\mathbb{R}) \ \text {and } 1+p(t)\mu(t)\neq 0, \ \forall \ t\in \mathbb{T}\Big\}\CommaBin
$$
and the set of positively regressive functions is denoted by
$$
\mathcal{R}^+:= \Big\{p\in C_{rd}(\mathbb{T};\mathbb{R}) \ \text {and } 1+p(t)\mu(t)> 0, \ \forall \ t\in \mathbb{T}\Big\}\cdot
$$
\begin{defn}
If $p\in\mathcal{R}$, then we define the generalized exponential function by
$$
e_p(t,s)=\exp\Big( \int_s^t \xi_{\mu(\tau)}(p(\tau))\Delta \tau \Big)\q \text{for } \ t,s \in \mathbb{T},
$$
with the cylinder transformation
  \begin{equation*}
\xi_h (z)= \begin{cases}
   \displaystyle\frac{\log(1+hz)}{h} &\text{if } h\neq 0 \\[12pt]
z & \text{if } h=0.
\end{cases}
\end{equation*}
\end{defn}

In the case $\mathbb{T}=\mathbb{R}$, the exponential function is given by
$$
e_p(t,s)= \exp\Big ( \int _s^t p(\tau)d\tau \Big),
$$
for $s,t \in \mathbb{R}$, where $p:\mathbb{R}\rightarrow\mathbb{R}$ is a continuous function. In the case  $\mathbb{T}=\mathbb{Z}$, the exponential is given by
$$
e_p(t,s)=\prod_{\tau=s}^{t-1}[1+p(\tau)],
$$
for $s,t \in \mathbb{Z}$, where $p:\mathbb{Z}\rightarrow\mathbb{R}$, $ p(t)\neq -1$ for all $t\in \mathbb{Z}$.\\
For more basic properties of the generalized exponential function, see \cite{Boh1}.

Next we define the upper and lower solutions of the integral equation \eqref{eq41}as follows:
\begin{defn}
A function $w\in C_{rd}(I_\mathbb{T},\mathbb{R})$ is said to be upper solution of \eqref{eq41} if
$$
w(t)\geq f(t)+\int_a^tk(t,s,w(s))\Delta s, \q \text{for all } t\in I_\mathbb{T},
$$
and similarly, a function $v\in C_{rd}(I_\mathbb{T},\mathbb{R})$ is said to be lower solution of \eqref{eq41} if
$$
v(t)\leq f(t)+\int_a^tk(t,s,v(s))\Delta s, \q \text{for all } t\in I_\mathbb{T}.
$$
\end{defn}
We define the sector $[v,w]$ as
$$
[v,w]=\{z\in C_{rd}(I_\mathbb{T},\mathbb{R}):\ v(t)\leq z(t)\leq w(t), \ t\in I_\mathbb{T}\}\cdot
$$
\begin{defn}
The functions $\alpha, \beta \in C_{rd}(I_\mathbb{T},\mathbb{R})$ are called maximal and minimal solutions of Equation \eqref{eq41}, respectively, if any solution $x \in [v,w]$ satisfies the relation $\beta(t)\leq x(t)\leq\alpha(t)$ for all $t\in I_\mathbb{T}$.
\end{defn}
 %$$
 %\Omega=\{(t,s,p): t,s\in I_\mathbb{T}^2,\ v(t)\leq p\leq w(t)\}.
 %$$
%The definition of the generalized monomials on time scales $h_k:\mathbb{T}\times \mathbb{T}\rightarrow \mathbb{R}$ are given as follows:
%\begin{equation*}
%h_{k}(t,s):=
%\begin{cases}
%1,&k=0,\\
%\displaystyle \int_{s}^{t}h_{k-1}(\eta,s)\Delta\eta,&k\in\mathbb{N}.
%\end{cases}
%\end{equation*}
%A relation between the exponential function and the monomials is given by
%$$
%e_\lambda(t,s)=\sum_{k=0}^\infty \lambda^kh_k(t,s)\q \text{for } s,t\in \mathbb{T}\ \text{with }t\geq s.
%$$

%\begin{thm}
%Let $k\in \mathbb{N}_0$. Then we have
%\begin{equation}\label{hhh}
%0\leq h_k(t,s)\leq \frac{(t-s)^k}{k!}, \q \text{for all } t\geq s.
%\end{equation}
%\end{thm}

\section{Existence Results}

  In this section we investigate the existence and uniqueness of solutions of Equation \eqref{eq41} in the sector $[v,w]$.

We need the result \cite [Theorem 5.9]{Kul} on which our discussion depends.
\begin{thm}\label{tisdel}
Consider the integral dynamic equation \eqref{eq41}. Let $k:I_\mathbb{T}^2\times\mathbb{R}^n\rightarrow\mathbb{R}^n$ and $f:I_\mathbb{T}\rightarrow\mathbb{R}^n$ be both continuous. If there exists a nonnegative constant $N$ such that
$$
\|k(t,s,p)\|\leq N, \q \forall \ (t,s)\in I_\mathbb{T}^2,\ p\in \mathbb{R}^n,
$$
then equation \eqref{eq41} has at least one solution.
\end{thm}
   Now, the existence of at least on solution of Equation \eqref{eq41} which lies in the sector $[v,w]$ is established in the following theorem.
  \begin{thm}Assume $k\in C_{rd}(I_\mathbb{T}^2\times\mathbb{R},\mathbb{R})$ and $f \in C_{rd}(I_\mathbb{T},\mathbb{R})$. If the functions $v,w$ are, respectively, lower and upper solutions of \eqref{eq41}, then \eqref{eq41} has at least one solution, $x$, such that $x\in[v,w]$ on  $I_\mathbb{T}$.
  \end{thm}

\begin{proof}
Define the modification of $k$ with respect to $v$ and $w$ for each fixed $(t,s,p)\in \Omega$ by
\begin{equation}\label{eq2}
 G(t,s,p)=
\begin{cases}
 k(t,s,v(s))+\displaystyle\frac{v(t)-p}{1+(v(t)-p)^2}, & \text{if }\ p < v(t);\\
 k(t,s,p),& \text{if } v(t)\leq p \leq w(t);\\
 k(t,s,w(s))-\displaystyle\frac{w(t)-p}{1+(w(t)-p)^2}, & \text{if }\ p > w(t).\\
\end{cases}
\end{equation}
We show that $G$ is bounded on $I_\mathbb{T}^2\times \mathbb{R}$. See that $k$ is rd-continuous function on compact region $\Omega:=\{(t,s,p): (t,s)\in I_\mathbb{T}^2 \ \text{and } v(t) \leq p \leq w(t)\}$  and so it is bounded. So we have
$$
|G(t,s,p)|< N, \qq \text{on }\ \Omega.
$$
 Since $f,v,w$ are rd-continuous functions, the left hand side of \eqref{eq2} is rd-continuous on $I_\mathbb{T}^2\times \mathbb{R}$. Hence by Theorem \ref{tisdel}, the integral dynamic equation
 \begin{equation}\label{modified}
   \varphi(t)=f(t)+\int_a^tG(t,s,\varphi(s))\Delta s,\q t\in I_\mathbb{T}
 \end{equation}
 has at least one solution, $x$ on $I_\mathbb{T}$ . Now we shall prove that
$$
v(t)\leq x(t) \leq w(t),\q \text{for all } t\in I_\mathbb{T}.
$$
We shall only prove $v(t)\leq x(t)$, while proving $x(t) \leq w(t)$ is similar arguments.
Assume the converse that $v(t)>x(t)$, then we have from \eqref{eq2}
\begin{align*}
0< \psi(t):&= v(t)-x(t)\\
& = v(t)-f(t)-\int_a^t\Big[k(t,s,v(s))+\displaystyle\frac{v(t)-x}{1+(v(t)-x)^2}\Big]\Delta s\\
& < v(t)-f(t)-\int_a^tk(t,s,v(s))\Delta s \leq 0,
\end{align*}
 which leads to a contradiction. Thus $v(t)\leq x(t)$ for all $t\in I_\mathbb{T}$. Hence, we get $v(t)\leq x(t)\leq w(t)$ for all  $t\in I_\mathbb{T}$. Therefore, $x(t)$ is a solution to \eqref{eq41} for $t\in I_\mathbb{T}$ and $v(t)\leq x(t)\leq w(t)$, which completes the proof.
\end{proof}
%====================================================================================================================

The following lemma that has been proved in \cite{nagumo} is important in estimating delta integrals in terms of regular integrals.
\begin{lem}
Let $h:\mathbb{R}\rightarrow \mathbb{R}$ be a continuous, nondecreasing function. If $\lambda_1,\lambda_2\in \mathbb{T}$ with $\lambda_1\leq\lambda_2$ then
$$
\int_{\lambda_1}^{\lambda_2}h(t)\Delta t\leq \int_{\lambda_1}^{\lambda_2}h(t)dt.
$$
\end{lem}

 Now we introduce a result which guarantees the existence of a unique solution to integral dynamic equation \eqref{eq41} by using successive approximation method.

  \begin{thm}
  Let $f \in C_{rd}(I_\mathbb{T},\mathbb{R})$ and $v,w$ are lower and upper solutions of equation \eqref{eq41}, respectively, with $v(t)\leq w(t)$ on $I_\mathbb{T}$. Assume further that $k\in C_{rd}(I_\mathbb{T}^2\times\mathbb{R},\mathbb{R})$ be non-decreasing function in the third argument and there exists $L\geq 0$ such that
\begin{equation}
|k(t,s,p)-k(t,s,q)|\leq L|p-q|, \q \forall \ t,s\in I_\mathbb{T}^2 \q \text{and } \ p,q\in \mathbb{R}.
\end{equation}
Then the successive approximations defined by
\begin{align*}
x_0(t) & = f(t)+\int_a^t k(t,s,v(s))\Delta s,\\
x_{n}(t):& = f(t) + \int _a^t k(t,s,x_{n-1}(s))\Delta s , \q n=1,2,...
\end{align*}
converge uniformly on $I_\mathbb{T}$ to the unique solution, $x$, of equation \eqref{eq41} such that $x\in[v,w]$ on $I_\mathbb{T}$.
   \end{thm}

\begin{proof}

 $1.$ \emph{Existence and Uniqueness}:\\
First, we show that  $\displaystyle\{x_n(t)\}_{n\in\mathbb{N}_0}$ converges uniformly on $I_\mathbb{T}$. To analyse the convergence of our sequence, we write $x_n(t)$ as a telescoping sum
\begin{equation}\label{eq03}
 x_n(t) = x _0(t)+\sum_{k=1}^{n} \ [x_{k}(t)-x_{k-1}(t)], \q \forall \ t\in I_\mathbb{T}\ \text{and } n\in \mathbb{N}.
\end{equation}
So
$$
\lim_{n\rightarrow\infty}= x _0(t)+\sum_{k=1}^{\infty} \ [x_{k}(t)-x_{k-1}(t)]
$$
For each integer $k\geq1$, we prove the following estimate by mathematical induction
\begin{equation}\label{eq4}
  |x_{k}(t)-x_{k-1}(t)|\leq M \frac{L^k(t-a)^k}{k!}, \q \forall \ t\in I_\mathbb{T};
\end{equation}
where $\displaystyle M:=\sup_{t\in I_\mathbb{T}}|x_0(t)-v(t)|$. For $k=1$ and for all $t\in I_\mathbb{T}$, we have
\begin{align*}
|x_1(t)-x_0(t)|&\leq  \int_a^t|[k(t,s,x_0(s))-k(t,s,v(s))]|\Delta s\\
& \leq L\int_a^t |x_0(s)-v(s)|\Delta s\\
& \leq LM(t-a).
\end{align*}
Hence the estimate \eqref{eq4} is valid for $k=1$. Now, assume that \eqref{eq4} is true for some $k=i>1$, then for all $t\in I_\mathbb{T}$ we have
\begin{align*}
|x_{i+1}(t)-x_{i}(t)|&\leq \int_a^t|[k(t,s,x_i(s))-k(t,s,x_{i-1}(s))]|\Delta s\\
&\leq L\int_a^t |x_i(s)-x_{i-1}(s)|\Delta s\\
&\leq L\int_a^tL^i M\frac{(s-a)^i}{i!}\Delta s\\
&\leq  M\frac{L^{i+1}}{i!}\int_a^t(s-a)^i\ ds\\
&= M\frac{L^{i+1}(t-a)^{i+1}}{(i+1)!}\cdot
\end{align*}
Thus, the estimate \eqref{eq4} is true for $k=i+1$ and so \eqref{eq4} holds for all integers $k\geq 1$. So for all $t\in I_\mathbb{T}$ the estimate \eqref{eq4} gives
\begin{align*}
|x_{k}(t)-x_{k-1}(t)|&\leq M \frac{L^k(t-a)^k}{k!}\\
&\leq M\frac{L^k(b-a)^k}{k!}\cdot
\end{align*}
See that
$$
\sum_{k=1}^\infty M\frac{L^k(b-a)^k}{k!}=-1+\sum_{k=0}^\infty M\frac{L^k(b-a)^k}{k!}=Me^{L(b-a)}-1.
$$

%$$
%\sum_{k=1}^\infty L^kMh_k(t,a)= \sum_{k=0}^\infty L^kMh_k(t,a)-M\\
%= M(e_L(t,a)-1)\leq M(e_L(b,a)-1)
%$$
It follows from the Weierstrass M-test that the infinite series
$$
\sum_{k=1}^{\infty} [x_{k}(t)-x_{k-1}(t)],
$$
converges uniformly on $t\in I_\mathbb{T}$. Thus, from \eqref{eq03} by letting $n\rightarrow \infty$, we see that the right hand side has a limit on $I_\mathbb{T}$. That means, the sequence $\{x_n(t)\}_{n\in \mathbf{N}_0}$ converges uniformly on $I_\mathbb{T}$ to some $x \in C_{rd}(I_\mathbb{T},\mathbb{R})$.

Next, we show that the limit $x$ is a solution to \eqref{eq41} on $I_\mathbb{T}$. We have for all $t\in I_\mathbb{T}$ and $n\geq0$
$$
\Big|\int_a^t [k(t,s,x_n)-k(t,s,x(s))]\Delta s\Big|\leq L \int_a^t |x_n(s)-x(s)|\Delta s,
$$
%\begin{align*}
%|\varphi_n(t)-\varphi(t)|&\leq \sum_{k=n+1}^\infty|\varphi_k(t)-\varphi_{k-1}(t)|\\
%& \leq M\sum_{k=n+1}^\infty L^kh_k(t,a)
%\end{align*}
by letting $n\rightarrow\infty$, we see that the right hand side of the above inequality goes to zero, thus we have
$$
\lim_{n\rightarrow\infty}\int_a^tk(t,s,x_n(s))\Delta s = \int_a^tk(t,s,x(s))\Delta s,
$$
which proves that $x$ is a solution of the integral equation \eqref{eq41}. To prove uniqueness of solutions, we show that any two solutions of equation \eqref{eq41} are necessarily identical. Assume there exist another solution $y(t)$ to the equation \eqref{eq41} on $I_\mathbb{T}$. For each $t\in I_\mathbb{T}$, we have
\begin{align*}
r(t):& = |x(t)-y(t)|\\
&\leq \int_a^t |k(t,s,x(s))-k(t,s,y(s))|\Delta s\\
&\leq L \int_a^t |x(s)-y(s)|\Delta s\\
&\leq L \int_a^t r(s)\Delta s.
\end{align*}
 Applying Gr\"onwall's inequality we get that $r(t)\leq0$ for all $t\in I_\mathbb{T}$, i.e., $r(t)\equiv0$ on $I_\mathbb{T}$. This proves uniqueness of solutions.

$2.$ \emph{The unique solution lies in the sector} $[v,w]$:

 Using mathematical induction we show that
\begin{equation}\label{sector}
 v(t)\leq x_n(t)\leq w(t), \q \forall \ t \in I_\mathbb{T}.
\end{equation}
 For $n=0$, we have for all $t\in I_\mathbb{T}$
 \begin{align*}
   v(t)\leq x_0(t)=& f(t)+\int_a^tk(t,s,v(s))\Delta s \\
   \leq & f(t)+\int_a^tk(t,s,w(s))\Delta s\leq w(t).
 \end{align*}
 So, $v(t)\leq x_0\leq w(t)$ for $t\in I_\mathbb{T}$. Assume for $n=i>0$ that $v(t)\leq x_i(t)\leq w(t)$ for $t\in I_\mathbb{T}$, then using the assumption that $v(t)\leq w(t)$ for $t\in I_\mathbb{T}$ and $k$ is non-decreasing in the third argument we have
 \begin{align*}
   x_{i+1}(t)=& f(t)+\int_a^tk(t,s,x_i(s))\Delta s \\
   \leq & f(t)+\int_a^tk(t,s,w(s))\Delta s\leq w(t).
 \end{align*}
 Similarly, we can show that $ v(t)\leq x_{i+1}(t)$ on $I_\mathbb{T}$. So we conclude that \eqref{sector} is true for all $n\geq0$ and all $t\in I_\mathbb{T}$. From the uniform convergence of the sequence $\{x_n(t)\}$ on $t\in I_\mathbb{T}$, we conclude that unique solution of Equation \eqref{eq41} satisfies
$$
 v(t)\leq x(t)\leq w(t), \q \forall \ t \in I_\mathbb{T}.
$$
The proof is complete.
\end{proof}

\section{Maximal and minimal solutions}
% we have to write some thing here about this method.
 Here, the method of monotone iterative technique is applied to investigate the existence of maximal and minimal solutions of integral equation \eqref{eq41}. %The upper and lower solutions are used to generate the closed set served as upper and lower bounds for solutions which can be improved by monotone iterative procedures.
\begin{thm}
Let $f \in C_{rd}(I_\mathbb{T},\mathbb{R})$ and $v,w $ are lower and upper solution of equation \eqref{eq41}, respectively, with $v(t)\leq w(t)$ on $I_\mathbb{T}$. Assume further that $k\in C_{rd}(I_\mathbb{T}^2\times\mathbb{R},\mathbb{R})$ be non-decreasing function in the third argument. Then the sequences $\{v_n\}$ and $\{w_n\}$ defined by
 \begin{equation}\label{lower}
  v_n(t)=f(t)+\int_a^t k(t,s,v_{n-1}(s))\Delta s,
 \end{equation}
  \begin{equation}\label{upper}
  w_n(t)=f(t)+\int_a^t k(t,s,w_{n-1}(s))\Delta s,
  \end{equation}
 for $t\in I_\mathbb{T}$ with $v_0=v$ and $w_0=w$, converge uniformly and monotonically to the maximal and minimal solutions $\alpha$ and $\beta$ of equation \eqref{eq41}, respectively.
\end{thm}
\begin{proof}
Define $\xi_1(t)=v_1(t)-v_0(t)$ for $t \in I_\mathbb{T}$, then
\begin{equation*}
\xi_1(t)\geq f(t)+\int_a^t k(t,s,v_0(s))\Delta s - f(t)-\int_a^t  k(t,s,v(s))\Delta s=0.
\end{equation*}
Thus $v(t)\leq v_1(t)$ on $I_\mathbb{T}$, as $v_0=v$ . Similarly, we can show that $w_1(t)\leq w(t)$ on $I_\mathbb{T}$. Let assume for some integer $k>0$ that $v_{k-1}(t)\leq v_k(t)$ on $I_\mathbb{T}$. Then setting
$$
\xi_k(t)=v_{k+1}(t)-v_k(t),
$$
we get
\begin{align*}
\xi_k(t)&\geq f(t)+\int_a^t k(t,s,v_k(s))\Delta s - f(t)-\int_a^t  k(t,s,v_{k-1}(s))\Delta s\\
&\geq \int_0^t  k(t,s,v_{k-1}(s))\Delta s - \int_0^t  k(t,s,v_{k-1}(s))\Delta s\\
&=0,
\end{align*}
which implies $v_k(t)\leq v_{k+1}(t)$ on $I_\mathbb{T}$. Similarly, we can prove that $w_{k+1}(t)\leq w_k(t)$. Hence it follows by induction $v_{n-1}(t)\leq v_n(t)$ and $w_n\leq w_{n-1}$ for all $n$ on $I_\mathbb{T}$. Now, define $\psi(t)=w_1(t)-v_1(t)$, then using the fact that $v(t)\leq w(t)$ for $t\in I_\mathbb{T}$ and $k$ is non decreasing in the third argument we have
  \begin{align*}
  \psi(t)&= f(t)+\int_a^tk(t,s,w_0(s))\Delta s - f(t)-\int_a^tk(t,s,v_0(s))\Delta s\\
  & \geq \int_a^tk(t,s,v(s))\Delta s - \int_a^tk(t,s,v(s))\Delta s\\
  &  = 0,
  \end{align*}
  which implies $v_1(t)\leq w_1(t)$ on $I_\mathbb{T}$. By following an induction argument we have $v_n\leq w_n$ for all $n$ on $I_\mathbb{T}$. We conclude from the previous argument that
  $$
  v=v_0\leq v_1\leq v_2\leq \cdots \leq v_n\leq w_n\leq \cdots \leq w_2\leq w_1 \leq w_0=w \q \text{on } I_\mathbb{T}.
  $$
Also  for $t_1,t_2\in I_\mathbb{T}$ we have
\begin{align*}
  |v_n(t_1)-v_n(t_2)|&=\Big|f(t_1)-f(t_2)+\int_a^{t_1}k(t_1,s,v_{n-1}(s))\Delta s - \int_a^{t_2}k(t_1,s,v_{n-1}(s))\Delta s\Big|  \\
   & \leq |f(t_1)-f(t_2)|+\Big|\int_a^{t_1}k(t_1,s,v_{n-1}(s))\Delta s - \int_a^{t_2}k(t_2,s,v_{n-1}(s))\Delta s\Big|\\
   & \leq|f(t_1)-f(t_2)|+\Big|\int_a^{t_1}[k(t_1,s,v_{n-1}(s))-k(t_2,s,v_{n-1}(s))]\Delta s\Big|\\
   & \q \q \q +\Big|\int_{t_1}^{t_2}k(t_2,s,v_{n-1}(s))\Delta s\Big|\rightarrow 0, \
   \q \text{as } \ |t_1-t_2| \rightarrow 0.
\end{align*}

    Thus $\{v_n\}$  is uniformly bounded and equicontinuous on $I_\mathbb{T}$. By Ascoli-Arzela theorem \cite{asc} we see that $\{v_n\}$ is relatively compact and hence there exists a subsequence of $\{v_n\}$ which converges uniformly on $I_\mathbb{T}$ to some $\alpha\in C_{rd}(I_\mathbb{T},\mathbb{R})$. Since the sequence $\{v_n\}$ is monotone, it converges uniformly to $\alpha$ too. By similar argument we prove that $\{w_n\}$ converges uniformly to some $\beta\in C_{rd}(I_\mathbb{T},\mathbb{R})$. It is easy to show that $\alpha$ and $\beta$ are solutions of equation \eqref{eq41} in view of the fact that $k$ is rd-continuous function and letting $n\rightarrow\infty$ in \eqref{lower}, \eqref{upper} we get
  $$
  \alpha(t)=f(t)+\int_a^t k(t,s,\alpha(s))\Delta s,
  $$
  and
  $$
  \beta(t)=f(t)+\int_a^t k(t,s,\beta(s))\Delta s,
  $$
for all $t\in I_\mathbb{T}$. Now we show that $\alpha$ and $\beta$ are maximal and minimal solution of \eqref{eq41}, respectively. Let $x$ be any solution of equation \eqref{eq41} such that $v(t)\leq x(t)\leq w(t)$ on $I_\mathbb{T}$. Set $r(t)=x(t)-v_1(t)$ for $t\in I_\mathbb{T}$. Since the function $k$ is nondecreasing in the third argument, so we have
\begin{equation*}
  r(t)= f(t)+\int_a^tk(t,s,x(s))\Delta s-f(t)-\int_0^tk(t,s,v(s))\Delta s\geq 0,
\end{equation*}
which implies $x(t)\leq v_1(t)$ on $I_\mathbb{T}$. Following by induction we can prove that $x(t)\leq v_n(t)$ for all $n\geq0$ on $I_\mathbb{T}$. Similarly, we can show that $x(t)\leq w_n(t)$ for all $n\geq0$ on $I_\mathbb{T}$. Thus we have $v_n(t)\leq x(t)\leq w_n(t)$ for all $n\geq0$ on $I_\mathbb{T}$. By taking the limit as $n\rightarrow\infty$, we obtain $\alpha(t)\leq x(t)\leq \beta(t)$ on $I_\mathbb{T}$. Then the proof is complete.

\end{proof}

\bibliographystyle{srtnumbered}

\end{document}